\documentclass[letterpaper, 10pt, conference]{ieeeconf}
\IEEEoverridecommandlockouts
\usepackage{cite,setspace}

\usepackage{amssymb,amsfonts,amsmath,amsthm}
\usepackage{mathtools}
\usepackage{enumerate}
\usepackage{verbatim}
\usepackage{commath}
\usepackage[mathcal]{euscript}
\usepackage[usenames]{color}
\usepackage{xcolor}

\definecolor{plum} {rgb}{.4,0,.4}
\definecolor{BrickRed} {rgb}{0.6,0,0}

\usepackage[plainpages=false,pdfpagelabels,colorlinks=true,linkcolor=BrickRed,citecolor=plum]{hyperref}

\def\ddefloop#1{\ifx\ddefloop#1\else\ddef{#1}\expandafter\ddefloop\fi}

\def\ddef#1{\expandafter\def\csname b#1\endcsname{\ensuremath{\boldsymbol{#1}}}}
\ddefloop ABCDEFGHIJKLMNOPQRSTUVWXYZabcdefghijklmnopqrstuvwxyz\ddefloop

\def\ddef#1{\expandafter\def\csname c#1\endcsname{\ensuremath{\mathcal{#1}}}}
\ddefloop ABCDEFGHIJKLMNOPQRSTUVWXYZ\ddefloop
 
\def\ddef#1{\expandafter\def\csname s#1\endcsname{\ensuremath{\mathsf{#1}}}}
\ddefloop ABCDEFGHIJKLMNOPQRSTUVWXYZ\ddefloop

\def\Reals{{\mathbb R}}
\def\Ex{{\mathbf E}} 
\def\trn{{\hbox{\it\tiny T}}} 
\def\eps{\varepsilon}
\def\deq{:=}

\newtheorem{theorem}{Theorem}

\newtheorem{remark}{Remark}

\title{\LARGE{\textbf{A Constructive Approach to Function Realization\\
by Neural Stochastic Differential Equations}}}

\author{Tanya Veeravalli and Maxim Raginsky%
\thanks{This work was supported in part by the NSF under award CCF-2106358 (``Analysis and Geometry of Neural Dynamical Systems'') and in part by the Illinois Institute for Data Science and Dynamical Systems
(iDS${}^2$), an NSF HDR TRIPODS institute, under award CCF-1934986.}%
\thanks{The authors are with the Department of Electrical and Computer Engineering and the Coordinated Science Laboratory, University of Illinois, Urbana, IL, USA. Emails: veerava2@illinois.edu, maxim@illinois.edu}
}

\begin{document}

\maketitle


\begin{abstract}
The problem of function approximation by neural dynamical systems has typically been approached in a top-down manner: Any continuous function can be approximated to an arbitrary accuracy by a sufficiently complex model with a given architecture. This can lead to high-complexity controls which are impractical in applications. In this paper, we take the opposite, constructive approach: We impose various structural restrictions on system dynamics and consequently characterize the class of functions that can be realized by such a system. The systems are implemented as a cascade interconnection of a neural stochastic differential equation (Neural SDE), a deterministic dynamical system, and a readout map. Both probabilistic and geometric (Lie-theoretic) methods are used to characterize the classes of functions realized by such systems.
\end{abstract}


\section{Introduction}

There is an extensive literature on function approximation by neural nets. For instance, a classical result due to Leshno et al.~\cite{Leshno_etal_NN} states the following: Let $\sigma : \Reals \to \Reals$ be a continuous function which is not a polynomial.  Then for any continuous $f : \Reals^n \to \Reals$, any compact $K \subset \Reals^n$, and any $\eps > 0$, there exist $N$ tuples $(c_i,a_i,b_i) \in \Reals \times \Reals^n \times \Reals$ such that
\begin{align}\label{eq:2LNN_approx}
	\sup_{x \in K} \Bigg|f(x) - \sum^N_{i=1} c_i \sigma(a_i^\trn x + b_i)\Bigg| \le \eps
\end{align}
(in fact, the universal approximation property holds only if $\sigma$ is not a polynomial). The approximating function in \eqref{eq:2LNN_approx} is an instance of a neural net with one hidden layer, $N$ hidden units (or neurons), and activation function $\sigma$. There are also universal approximation results for multilayer neural nets \cite{Yarotsky_ReLU_approx}, as well as depth separation results \cite{Telgarsky_depth_sep} which show (constructively) that there exist functions that admit efficient representation by multilayer neural nets, yet require exponentially more hidden units for representation by ``shallower'' nets.

Recently, motivated in large part by the widespread use of deep learning, there has been a lot of interest in continuous-time abstractions of neural nets \cite{Hirsch89neural,Haber_2018,Weinan2017dynamical,chen2018neural}, aptly termed \textit{neural ODEs}. These are modeled by dynamical systems of the form\begin{subequations}\label{eq:neural_ODE}
	\begin{align}
	\dot{z}(t) &= f(z(t),w(t)), \qquad z(0) = \alpha(x)\\
	y &= \beta(z(1))
	\end{align}
\end{subequations}
where the $q$-dimensional input $x$ is mapped to an initial $n$-dimensional state $z(0)$ by a read-in map $\alpha : \Reals^q \to \Reals^n$, the $p$-dimensional output $y$ is produced from the state $z(1)$ at time $t=1$ by a read-out map $\beta : \Reals^n \to \Reals^p$, and $f(\cdot,w)$ is a family of vector fields on $\Reals^n$ parametrized by  $w \in \Reals^m$. For example, we could have something like $w = (C,A,b) \in \Reals^{n \times k} \times \Reals^{k \times n} \times \Reals^k$ and
\begin{align*}
	f(x,w) = C\boldsymbol{\sigma}(Ax + b),
\end{align*}
where $\boldsymbol{\sigma}$ is the diagonal map
\begin{align}\label{eq:diag_sigma}
	\boldsymbol{\sigma}(z) \deq \big(\sigma(z_1),\dots,\sigma(z_k)\big)^\trn, \, z \in \Reals^k.
\end{align}
The problem of universal approximation is then to establish the conditions under which any continuous function $f : \Reals^q \to \Reals^p$ could be approximated, to any given accuracy $\eps$, on a given compact set $K \subset \Reals^q$ by a system of the form \eqref{eq:neural_ODE} for a suitable choice of the control law $w : [0,1] \to \Reals^m$. 

This problem has been addressed in a number of recent works \cite{li2019deep,Ruiz-Balet:2021yj,tabuada2021universal,Agrachev2022}. While the methods and techniques differ, the overall philosophy of the results is \textit{top-down}---under suitable structural assumptions on $\alpha$, $\beta$, and $f(\cdot,w)$, \textit{any} continuous function can be approximated to \textit{any} desired accuracy by choosing a suitable control $w(\cdot)$. The resulting controls, however, tend to exhibit fairly high complexity unlikely to be tolerated in practical applications. For example, they may be piecewise constant, but the number of pieces will grow like $(1/\eps)^d$, where $d$ depends polynomially on $q$ \cite{li2019deep,Ruiz-Balet:2021yj,tabuada2021universal}; other constructions make use of high-gain, high-frequency controls \cite{Agrachev2022}, with similar complexity issues.

\subsection{A constructive approach}

In this paper, we approach the problem of function realization by neural dynamical systems from a \textit{constructive}, \textit{bottom-up} perspective. That is, instead of fixing \textit{a priori} the class of functions to be approximated, we impose various structural restrictions on the system dynamics and then characterize the class of functions that can be realized by the system. To fix ideas, let us consider a cascade interconnection \cite{Krener_decomposition} of two dynamical systems of the following form: 
\begin{subequations}\label{eq:cascade_ODE}
\begin{align}
	\dot{w}(t) &= g(w(t),\theta) \label{eq:cascade_ODE_weights}\\
	\dot{z}(t) &= f(z(t),w(t)), \label{eq:cascade_ODE_activations}
\end{align}
\end{subequations}
where $g(\cdot,\theta)$ is a family of vector fields on $\Reals^m$ parametrized by $\theta \in \Reals^k$. This system differs from \eqref{eq:neural_ODE} in one key respect: The trajectory $w(t)$ is not prescribed externally, but is instead generated internally according to \eqref{eq:cascade_ODE_weights}. Given the read-in map $\alpha$ and the read-out map $\beta$, as in \eqref{eq:neural_ODE}, we can realize different functions of $x \in \Reals^q$ by tuning the parameter $\theta$ [and, possibly, the initial condition $w(0)$]. 

Neural ODE models of this type have been considered in the literature. For example, Choromanski et al.~\cite{ODEtoODE} analyzed the case when $w$ takes values in the orthogonal group $O(n)$, the vector fields $g(\cdot,\theta)$ are of the form
\begin{align*}
	g(w,\theta) = w\tilde{g}(w,\theta)
\end{align*}
for some finitely parametrized map $\tilde{g}(\cdot,\theta) : \Reals^{n \times n} \to {\rm Skew}(n)$, where ${\rm Skew}(n)$ is the set of all $n \times n$ skew-symmetric matrices with real entries, and 
\begin{align*}
	f(z,w) = \boldsymbol{\sigma}(wz),
\end{align*}
where $\boldsymbol{\sigma}$ is the coordinatewise application of some scalar nonlinearity $\sigma : \Reals \to \Reals$, as in \eqref{eq:diag_sigma}. While Choromanski et al.\ did not explicitly characterize the class of functions that can be realized by this model, they showed that they enjoy certain stability properties that are advantageous in applications.

We will use cascade models like \eqref{eq:cascade_ODE} as a starting point, but will replace the deterministic dynamics \eqref{eq:cascade_ODE_weights} with a stochastic differential equation (SDE) driven by a multidimensional Brownian motion.  The rationale for using such \textit{neural SDEs} \cite{Wong_1991,Tzen2019TheoreticalGF,tzen2019neural,pmlr-v108-li20i} is twofold: First, by using Brownian motion processes as \textit{generalized inputs} \cite{Sussmann_geninputs}, we will be able to generate all the internal complexity we need while only varying the global parameters of the SDE. Second, it provides us with a sampling-based mechanism for constructing finite approximations \cite{barron1993universal}. Indeed, if $F(x;\omega)$ is a random function of $x$, then the expected value $f(x) \deq \Ex_\omega[F(x;\omega)]$ can be closely approximated by finite sums of the form
\begin{align*}
	\hat{f}_N(x) = \frac{1}{N}\sum^N_{i=1}F(x;\omega_i),
\end{align*}
where $\omega_1,\dots,\omega_N$ are independent copies of $\omega$, with the $L^2$ approximation error $\Ex[(f(x)-\hat{f}_N(x))^2]$ scaling as $O(1/N)$ under suitable assumptions on the second moment of $F(x;\omega)$. (While finite approximation is not the focus of this paper, our earlier work \cite{Veeravalli_L4DC2023} contains a discussion of sampling in the context of neural SDEs.)

\subsection{Previous results}

In an earlier work \cite{Veeravalli_L4DC2023}, we have obtained some constructive realizability results for a certain class of neural SDEs. Briefly, we have considered $n$-dimensional It\^o SDEs 
\begin{align}\label{eq:neural_SDE_L4DC}
	\dif X_t &= a(X_t; \theta) \dif t + b(X_t; \theta) \dif V_t, \qquad X_0 = x 
\end{align}
with finitely parametrized drift and diffusion coefficients satisfying the uniform ellipticity condition---the eigenvalues of the matrix $b(x;\theta)b(x;\theta)^\trn$ are uniformly bounded away from $0$.  Let $v$ be a fixed (nonrandom) vector in $\Reals^n$. We say that, for each choice of $\theta$, the above system \textit{realizes} the function $f(x;\theta) \deq \Ex[Y] = \Ex[v^\trn X_1]$ of the initial condition $x$. We have shown that $f(x;\theta)$ can be approximated by
\begin{align*}
	f(x;\theta) \approx c_1\int_{\Reals^n}  v^\trn z \cdot \exp\big(-c_2 I(x,z;\theta)\big) \dif z,
\end{align*}
where $c_1,c_2$ are some constants and where $I(x,z; \theta)$ is the minimum of the control energy $\frac{1}{2}\int^1_0 |v(t)|^2 \dif t$ among all sufficiently regular (e.g., $L^2$) controls $v : [0,1] \to \Reals^n$ that transfer the state of the deterministic control-affine system
\begin{align}\label{eq:neural_ODE_L4DC}
	\dot{x}(t) = a(x(t);\theta) + b(x(t);\theta)v(t)
\end{align}
from $x(0) = x$ to $x(1) = z$. Thus, the problem of characterizing the class of functions realized in this way reduces to analyzing deterministic minimum-energy control problems, and some upper and lower bounds on $I(x,z;\theta)$ are given in \cite{Veeravalli_L4DC2023}. However, the assumption of uniform ellipticity is rather restrictive and, in particular, rules out the interesting cases when the diffusion process in \eqref{eq:neural_SDE_L4DC} is degenerate, yet the deterministic system \eqref{eq:neural_ODE_L4DC} is nevertheless completely controllable \cite{Elliott_diffusions,Brockett_spheres}.

\section{The basic model and its properties}

We start by presenting a relatively simple model which will allow for a clean analysis. It has the following form:
\begin{subequations}\label{eq:cascade_neural_SDE}
\begin{align}
	\dif W_t&= a(W_t;\theta) \dif t + b(W_t;\theta) \circ \dif V_t \label{eq:SDE_weights}\\
	\dif Z_t &= h(W_t,t)Z_t \dif t \label{eq:SDE_FK}\\
	 Y &= Z_1 \sigma(W_1^\trn x) \label{eq:SDE_readout}
\end{align}
\end{subequations}
Here, \eqref{eq:SDE_weights} is a Stratonovich SDE for an $n$-dimensional diffusion process $(W_t)$ with a deterministic initial condition $W_0 = w_0$.  It is driven by an $m$-dimensional Brownian motion $(V_t)$, and its drift and diffusion coefficients are parametrized by a $k$-dimensional vector of parameters $\theta$. We assume throughout that the drift and the diffusion coefficients are sufficiently well-behaved (e.g., Lipschitz) to guarantee existence and uniqueness of strong solutions of the SDE. The process $(Z_t)$ is scalar, initialized at $Z_0 = 1$, and $h : \Reals^n \times [0,1] \to \Reals$ is a continuous function. The real-valued output $Y$ is obtained by multiplying $Z_1$ by $\sigma(W_1^\trn x)$.

\begin{remark} {\em By adding an equation $\dif X_t = 0$ initialized at $X_0 = x$, we can view the above system as a stochastic variant of \eqref{eq:cascade_ODE} with internal state $(W_t,Z_t,X_t)$, read-in map
	\begin{align*}\alpha(x) = (w_0,1,x),
	\end{align*}
	and read-out map
	\begin{align*}\beta(w,z,x) = z\sigma(w^\trn x).
	\end{align*}}
\end{remark}

The output $Y$ of \eqref{eq:cascade_neural_SDE} is a \textit{random} function of $x$. We take its expected value as the function realized by this system:
\begin{align}\label{eq:f_theta}
	f(x;\theta) \deq \Ex[Z_1\sigma(W^\trn_1x)],
\end{align}
where in the left-hand side we have explicitly indicated the dependence of this function on the parameters $\theta$ of the SDE \eqref{eq:SDE_weights}. The class of functions realized in this way can be characterized as follows:

\begin{theorem}\label{thm:neural_SDE} The function \eqref{eq:f_theta} realized by the system \eqref{eq:cascade_neural_SDE} can be expressed as
	\begin{align}\label{eq:f_theta_FK}
		f(x;\theta) = \Ex\Bigg[\exp\Bigg(\int^1_0 h(W_t,t)\dif t\Bigg)\sigma(W^\trn_1 x)\Bigg].
	\end{align}
Moreover, let $\cL^\theta$ be the second-order linear differential operator
	\begin{align*}
		&\cL^\theta \varphi(w) \deq \tilde{a}(w;\theta)^\trn \nabla \varphi(w) + \frac{1}{2}{\rm tr} \{ b(w;\theta)b(w;\theta)^\trn \nabla^2 \varphi(w)\}
	\end{align*}
	where
	\begin{align*}
		\tilde{a}(w;\theta) = a(w;\theta) + \frac{1}{2}\sum^m_{i=1} \frac{\partial b_i}{\partial w}(w; \theta) b_i(w; \theta)
	\end{align*}
	is the It\^o-corrected drift, and where $b_i(w;\theta)$ is the $i$th column of $b(w;\theta)$. Consider the PDE
	\begin{align}\label{eq:PDE_FK}
		\frac{\partial u}{\partial w}(w,t;x,\theta) = \cL^\theta u(w,t;x,\theta) + h(w,t)u(w,t;x,\theta)
	\end{align}
with the initial condition $u(w,0; x,\theta) = \sigma(w^\trn x)$. If the function $w \mapsto \sigma(w^\trn x)$ is in the domain of $\cL^\theta$, then 
\begin{align*}
	f(x;\theta) = u(w_0,1; x,\theta).
\end{align*}
\end{theorem}
\begin{proof} The equation \eqref{eq:SDE_FK} for $Z_t$ can be solved for each trajectory $(W_t)$:
	\begin{align*}
		Z_t &= Z_0 \exp\Bigg(\int^t_0 h(W_s,s)\dif s \Bigg) \nonumber\\
		&= \exp\Bigg(\int^t_0 h(W_s,s)\dif s \Bigg).
	\end{align*}
	Substituting this into the expression for $Y$ and taking expectations, we obtain \eqref{eq:f_theta_FK}.
	
	The operator $\cL^\theta$ is the infinitesimal generator of the diffusion process $(W_t)$ governed by the SDE \eqref{eq:SDE_weights}. Since $w \mapsto \sigma(w^\trn x)$ is in the domain of $\cL^\theta$, the function
	\begin{align*}
		& u(w,t;x,\theta) \nonumber\\
		&= \Ex\Bigg[\exp\Bigg(\int^t_0 h(W_s,s)\dif s\Bigg)\sigma(W^\trn_tx)\Bigg|W_0 = w\Bigg]
	\end{align*}
	is a solution of \eqref{eq:PDE_FK} subject to the initial condition $u(w,0;x, \theta) = \sigma(w^\trn x)$ by the (converse of) the Feynman--Kac theorem \cite[Thm.~21.1]{Kallenberg_probability}.
\end{proof}
\noindent A more elaborate set-up, in the spirit of \cite{ODEtoODE}, is as follows:
\begin{subequations}\label{eq:SDE_to_ODE}
\begin{align}
	\dif W_t &= a(W_t;\theta) \dif t + b(W_t;\theta) \circ \dif V_t \label{eq:SDE_to_ODE_weights}\\
	\dif Z_t &= \boldsymbol{\sigma}(W_t Z_t)\dif t \label{eq:SDE_to_ODE_state}\\
	Y &= v^\trn Z_1. \label{eq:SDE_to_ODE_readout}
\end{align}
\end{subequations}
The processes $(W_t)$ and $(Z_t)$ in Eqs.~\eqref{eq:SDE_to_ODE_weights} and \eqref{eq:SDE_to_ODE_state} are now evolving in $\Reals^{n \times n}$ and $\Reals^n$, respectively, and $v$ in \eqref{eq:SDE_to_ODE_readout} is a fixed vector in $\Reals^n$. With the initial condition $W_0 = w_0$ fixed, we introduce the read-in map $\alpha(x) = (w_0,x)$ and the read-out map $\beta(w,z) = v^\trn z$. The stochastic dynamics in \eqref{eq:SDE_to_ODE_weights} generates the matrices of weights $W_t$, which are used to control the dynamics of internal activations $Z_t$ in \eqref{eq:SDE_to_ODE_state}. The cascade form of the system allows us to first generate the random trajectory $(W_t)$ and then, conditionally on that, solve the ODE \eqref{eq:SDE_to_ODE_state}. As before, we say that the model \eqref{eq:SDE_to_ODE} realizes the function $f(x;\theta) = \Ex[Y] = \Ex[v^\trn Z_1]$. In contrast to \eqref{eq:cascade_neural_SDE}, where we were able to characterize the class of realized functions in a relatively clean manner, the determination of $(Z_t)$ for a given realization $(W_t)$ involves solving a time-inhomogeneous ODE
\begin{align}\label{eq:Z_ODE}
	\frac{\dif Z_t}{\dif t} = g_{W_t}(Z_t), \, Z_0 = x
\end{align}
where we have defined
\begin{align}\label{eq:gW}
	g_{W_t}(z) \deq \boldsymbol{\sigma}(W_tz).
\end{align}
We can compactly express the function realized by \eqref{eq:SDE_to_ODE} using the chronological exponential notation \cite{agrachev_sachkov_2004} as
\begin{align*}
	f(x;\theta) = \Ex\Bigg[v^\trn \overrightarrow{\exp}\left(\int^1_0 g_{W_t} \dif t\right)x \Bigg],
\end{align*}
which is just shorthand for the dependence of the solution of \eqref{eq:Z_ODE} on the initial condition $x$ and on the trajectory $(W_t)$, but any further analysis would involve asymptotic expansions in terms of iterated integrals and Lie derivatives.

\subsection{The role of Lie theory}

The above construction is quite general as it does not specify the form of the drift and the diffusion coefficients $a(w;\theta)$ and $b(w;\theta)$ in \eqref{eq:SDE_weights} and in \eqref{eq:SDE_to_ODE_weights}. \textit{A priori} we would expect that, in order to make the model \eqref{eq:cascade_neural_SDE} sufficiently expressive (i.e., to realize a sufficiently rich class of functions), we would need $a$ and $b$ to depend nonlinearly on both $w$ and $\theta$. This intuition, however, is misleading since we can let the functions $h$ and $\sigma$ do all the ``nonlinear work'' when the Lie algebra generated by the vector fields $a(\cdot;\theta), b_1(\cdot;\theta), \dots, b_m(\cdot;\theta)$ is finite-dimensional.

Thus, let $\cA$ be a finite-dimensional Lie algebra of vector fields on $\Reals^n$ and let $g_1,\dots,g_d$ be a basis of $\cA$, $d = \dim \cA$. For $\theta = (\theta_{ij})_{i = 0,\dots,m, j = 1,\dots d}$, take
\begin{align*}
	a(w;\theta) = \sum^d_{j=1} \theta_{0j} g_j(w)
\end{align*}
and
\begin{align*}
	b_i(w;\theta) = \sum^d_{j=1} \theta_{ij} g_j(w), \qquad i = 1, \dots, m.
\end{align*}
Since $\cA$ is finite-dimensional, it is isomorphic to a subalgebra of ${\rm gl}(\ell,\Reals)$ (the Lie algebra of $\ell \times \ell$ real matrices with the commutator bracket $[A,B] = AB - BA$) for some finite $\ell$ by Ado's theorem \cite{Varadarajan_Lie_groups}. Thus, without loss of generality (and by changing $n$ if necessary) we may take $g_i(w) = G_i w$ for some $G_i \in \Reals^{n \times n}$, so \eqref{eq:SDE_weights} will take the form
\begin{align}\label{eq:SDE_linparam}
	\dif W_t = A(\theta) W_t \dif t + \sum^m_{i=1} B_i(\theta)W_t \circ \dif V^i_t,
\end{align}
where $V^1_t,\dots,V^m_t$ are $m$ independent scalar Brownian motion processes, and where
\begin{align*}
	A(\theta) &= \sum^d_{j=1}\theta_{0j} G_j, \\
	B_i(\theta) &= \sum^d_{j=1}\theta_{ij} G_j, \quad i = 1, \dots, m.
\end{align*}
As an illustration of the expressive capabilities of this type of linear parametrization, we can consider Brockett's construction of a diffusion process on the sphere \cite{Brockett_spheres}: Let $A,B_1,\dots,B_m$ be $n \times n$ skew-symmetric matrices and consider the $n$-dimensional Stratonovich SDE
\begin{align}\label{eq:Brockett_SDE}
	\dif W_t = AW_t \dif t + \sum^m_{i=1} B_i W_t \circ \dif V^i_t.
\end{align}
or the equivalent It\^o SDE
\begin{align*}
	\dif W_t = \Bigg(A + \frac{1}{2}\sum^m_{i=1}B^2_i\Bigg)W_t \dif t + \sum^m_{i=1} B_i W_t \dif V^i_t.
\end{align*}
Applying It\^o's rule gives
\begin{align*}
	\dif\,(W^\trn_t W_t) &= 
	W^\trn_t \Bigg(A+A^\trn + \sum^m_{i=1}B^2_i\Bigg) W_t \dif t \nonumber\\
	&\qquad + \sum^m_{i=1} W^\trn_t (B_i + B^\trn_i) W_t  \dif V^i_t \\
	&\qquad + \sum^m_{i=1} W^\trn_t B^\trn_i B_i W_t \dif t \\
	& = 0
\end{align*}
so the Euclidean norm $|W_t|$ stays constant for all $t$. Thus, if we choose $w_0 \in S^{n-1}$, then the random trajectory $(W_t)_{t \ge 0}$ will be confined to $S^{n-1}$. The algebra ${\rm Skew}(n)$ of $n \times n$ skew-symmetric matrices has dimension $d = n(n-1)/2$, so we can express \eqref{eq:Brockett_SDE} in the form \eqref{eq:SDE_linparam} with a vector $\theta$ of $k = (m+1)n(n-1)/2$ parameters.

Similar considerations apply to the matrix-valued case. For example, if $A,B_1,\dots,B_m$ are $n \times n$ skew-symmetric matrices and the matrix Stratonovich SDE
\begin{align}\label{eq:Brockett_SDE_2}
	\dif W_t = AW_t \dif t + \sum^m_{i=1}B_i W_t \circ \dif V^i_t,
\end{align}
is initialized with $W_0 = w_0 \in O(n)$, then all $W_t$ will take values in $O(n)$ as well \cite{Brockett_spheres}. 

With regards to the role of $\sigma$, let us consider the family $\cF$ of the vector fields $g_W(z)$ defined in \eqref{eq:gW}. With typical choices of the nonlinearity $\sigma : \Reals \to \Reals$ (e.g., the hyperbolic tangent $\sigma(r) = \tanh r$), the Lie algebra generated by $\cF$ will be infinite-dimensional. To see this, suppose that $\sigma$ is $C^\infty$. Let $g,g'$ denote the vector fields $g_W,g_{W'}$ for two matrices $W,W' \in \Reals^n$. Then, using the formula
\begin{align*}
	\frac{\partial}{\partial z} \boldsymbol{\sigma}(Wz) = \boldsymbol{\sigma}^{(1)}(Wz)W
\end{align*}
for the Jacobian of $g_W$, where
$$
\boldsymbol{\sigma}^{(1)}(z) \deq {\rm diag}\big(\sigma'(z_1),\dots,\sigma'(z_n)\big), 
$$
a straightforward but tedious computation of the iterated Lie brackets
\begin{align*}
	{\rm ad}^{k+1}_g g' = {\rm ad}_g ({\rm ad}^{k}_g g'), \quad k = 0, 1, \dots
\end{align*}
where ${\rm ad}^0_g g' \deq g'$ and ${\rm ad}_g g' \deq [g,g'],$ yields vector fields whose coordinates involve derivatives of $\sigma$ of arbitrary orders. As an example, consider the case $\sigma(r) = \tanh r$. Since the derivative of $\sigma$ satisfies the relation $\sigma'(r) = 1-\sigma^2(r)$, the $k$th iterated Lie bracket ${\rm ad}^k_g g'$ will involve polynomials of degree $k+1$ in the entries of $\boldsymbol{\sigma}(Wz)$ and $\boldsymbol{\sigma}(W'z)$. Hence, using the linear independence of the monomials $1,r,r^2,\dots$, it is easy to see that the Lie algebra generated by $\cF$ will be infinite-dimensional. Moreover, the Lie algebra generated by $\cF$ may be infinite-dimensional even if $\sigma$ is a polynomial. Consider, for example, $\sigma(r) = 1 + r^3$; then even for $n=1$ the Lie algebra generated by $\cF$ will contain polynomial vector fields of arbitrary degree (and thus any continuous function can be approximated arbitrarily well on any given compact set by some element of this Lie algebra \cite{Agrachev2022}). This is in sharp contrast with the result of Leshno et al.~\cite{Leshno_etal_NN} on the necessity of nonpolynomial continuous  activation functions for universal approximation of continuous functions by neural nets with one hidden layer.

\subsection{Elaborations and extensions}

The basic model in \eqref{eq:cascade_neural_SDE} can be extended in various ways. For instance, we can consider matrix-valued processes $W_t$ and replace the output equation \eqref{eq:SDE_readout} with
\begin{align*}
	Y = Z_1 v^\trn \boldsymbol{\sigma}(W_1 x),
\end{align*}
where $W_t$ takes values in $\Reals^{p \times n}$ and $v \in \Reals^p$ is a fixed  vector. Another possibility is to let $W_t = (U_t,\tilde{W}_t)$ take values in $\Reals^p \times \Reals^{p \times n}$ and let
\begin{align*}
	Y = Z_1 U_1^\trn \boldsymbol{\sigma}(\tilde{W}_1 x).
\end{align*}
If $h \equiv 0$ in \eqref{eq:SDE_FK}, then the class of functions realized in this way is the closed convex hull of neural nets with one hidden layer and with arbitrarily many hidden units \cite{Gurvits97approximation}. To see this, let $(U^\nu,\tilde{W}^\nu)$, $\nu = 1,\dots,N$, be independent copies of $(U_1,\tilde{W}_1)$. (In other words, we just run $N$ independent copies of the stochastic dynamics \eqref{eq:SDE_weights} in parallel.) Then, assuming $U_t$ evolves on the unit sphere $S^{n-1}$, $\tilde{W}_t$ evolves on the orthogonal group $O(n)$, and the nonlinearity $\sigma : \Reals \to \Reals$ is bounded, the function
\begin{align*}
	\hat{f}_N(x) \deq \frac{1}{N}\sum^N_{\nu = 1} (U^\nu)^\trn \boldsymbol{\sigma}(W^\nu x)
\end{align*}
will be a good approximation of $f(x) = \Ex[U^\trn_1 \boldsymbol{\sigma}(W_1x)]$, for $N$ sufficiently large, by the law of large numbers. Fig.~\ref{fig:predict_neuron} shows a simple illustration of the ability of such models to realize neuron-like functions of the form $f(x) = \sigma(w^\trn x)$.
\begin{figure}[t]
	\includegraphics[width=\columnwidth]{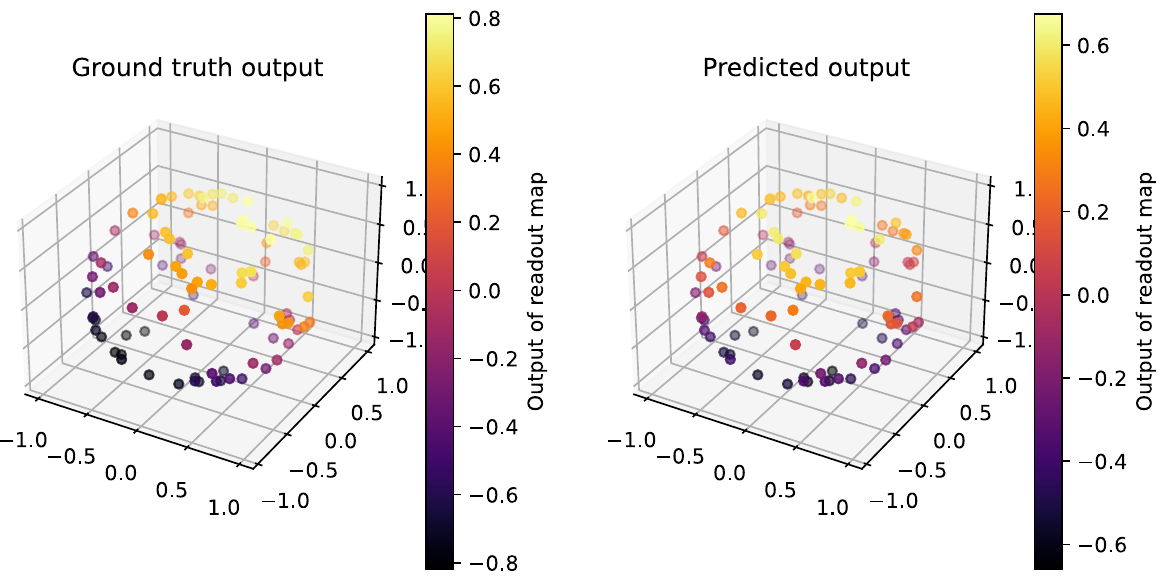}
	{\small \caption{Left: Sample values of the single-neuron function $f(x) = \tanh(w^\trn_0 x)$ on the unit sphere $S^2$, where $w_0 \in S^2$ is a fixed unit vector. Right: The values of $f(x)$ predicted by a model of the form \eqref{eq:cascade_neural_SDE} trained on random samples of $(X,f(X))$ with $X$ drawn from the uniform distribution on $S^2$. The weights $W_t$ evolve on $S^2$ according to \eqref{eq:Brockett_SDE} with $m = 2$.}\label{fig:predict_neuron}}
\end{figure}

The role of the function $h$ in \eqref{eq:SDE_FK} is to bias or regularize the trajectories $(W_t)$ in some manner, as far as their contribution to the value of $f(x;\theta)$ goes. For example, if some reference trajectory $(\xi_t)_{0 \le t \le 1}$ is given, we could take 
\begin{align*}
	h(w,t) = - |w - \xi_t|^2.
\end{align*}
In that case, we have
\begin{align*}
	f(x;\theta) = \Ex\Bigg[\sigma(W^\trn_1 x)\exp\Bigg(-\int^1_0 |W_t - \xi_t|^2\dif t \Bigg)\Bigg]
\end{align*}
which has the effect of penalizing those trajectories $(W_t)$ that differ too much from the reference.

The class of functions realized by the model \eqref{eq:SDE_to_ODE} will be in general richer than that realized by \eqref{eq:cascade_neural_SDE}. By way of illustration, consider the SDE \eqref{eq:Brockett_SDE_2} for a diffusion process on the orthogonal group $O(n)$. Suppose that the skew-symmetric matrices $A,B_1,\dots,B_m$ are such that the deterministic system
\begin{align}\label{eq:sphere_ODE}
	\dot{w}(t) = \Bigg(A + \sum^m_{i=1}v_i(t)B_i\Bigg) w(t)
\end{align}
is controllable on $O(n)$ \cite{Brockett_spheres}, i.e., any two matrices $w_0,w_1 \in O(n)$ can be joined by a curve lying along the trajectory $w(t)$ of \eqref{eq:sphere_ODE} generated by some piecewise constant controls $v_1(t),\dots,v_m(t)$. (For this, it suffices to consider the case $w_0 = I_n$ \cite{Jurdjevic_control_on_Lie_groups}.) Now, the probability law of the random path $(W_t)_{t \ge 0}$ in $O(n)$ starting at $W_0 = w_0$ is supported on the closure of the set of all trajectories of \eqref{eq:sphere_ODE} with $w(0) = w_0$ generated by piecewise constant controls \cite{Elliott_diffusions}, which by controllability covers the entire $O(n)$. This, in turn, will allow us to sample high-complexity trajectories for controlling \eqref{eq:Z_ODE}.

Finally, some comments on the generality of the considered models are in order. Both \eqref{eq:cascade_neural_SDE} and \eqref{eq:SDE_to_ODE} have a cascade structure, where the stochastic dynamics of the weights $W_t$ is autonomous (for each choice of the parameters $\theta$), and then these weights are used as generalized inputs \cite{Sussmann_geninputs} to the controlled dynamics of the internal activations $Z_t$. Moreover, the parameters $\theta$ explicitly enter only the dynamics of $W_t$. However, we can consider other models of neural SDEs discussed in the literature, for example a stochastically excited continuous-time recurrent neural net
\begin{align*}
	\dif X_t &= \big(- X_t + \boldsymbol{\sigma}(X_t)\big)\dif t + \sum^m_{i=1} B_i X_t \circ \dif V^i_t
\end{align*}
parametrized by $m$ matrices $B_1,\dots,B_m$. In models of this type, there is no explicit separation of the internal state into weights and activations. However, we can use the results of Freedman and Willems \cite{freedman1978smooth} and Krener and Lobry \cite{krener1981complexity} to show that such neural SDEs  can be \textit{simulated}, in a certain sense, by systems of the cascade type.

To that end, consider the $n$-dimensional stochastic system\begin{subequations}\label{eq:general_neural_SDE}
\begin{align}
	\dif X_t &= f(X_t)\dif t + g(X_t; \theta) \circ \dif V_t \\
	Y_t &= v^\trn X_t
\end{align}
\end{subequations}
with initial condition $X_0 = x$, where the columns of the $n \times m$ matrix $g(x;\theta)$ are smooth vector fields on $\Reals^n$ with a smooth dependence on a $k$-dimensional parameter vector $\theta$. Then we have the following result:

\begin{theorem} Suppose that the Lie algebra $\cA$ generated by $g_1(\cdot;\theta),\dots,g_m(\cdot;\theta)$, $\theta \in \Reals^k$, has finite dimension $d$. Then there exist smooth vector fields $b_1,\dots,b_d$ on $\Reals^d$, smooth functions $\beta_{ij} : \Reals^k \to \Reals$, $1 \le i \le m, 1 \le j \le d$, $h : \Reals^n \times \Reals^d \to \Reals^n$, $\varphi : \Reals^d \times \Reals^n \to \Reals^n$, and an almost surely positive stopping time $\tau$, such that the cascade system\begin{subequations}\label{eq:KL_cascade}
	\begin{align}
		\dif W_t &= \sum^m_{i=1}\Bigg(\sum^d_{j=1}\beta_{ij}(\theta)b_j(W_t)\Bigg)  \circ \dif V^i_t \label{eq:KL_cascade_weights}\\
		\dif Z_t &= h(Z_t,W_t) \dif t 
	\end{align}	
	\end{subequations}
	with $W_0 = 0$, $Z_0 = x$ simulates \eqref{eq:general_neural_SDE} up to time $\tau$, i.e., 
	\begin{align*}
		Y_t = v^\trn \varphi(W_t,Z_t), \qquad 0 \le t < \tau.
	\end{align*}
\end{theorem}
\begin{remark} {\em The above theorem shows that we can simulate a system like \eqref{eq:general_neural_SDE}, at least locally, by a cascade system like \eqref{eq:KL_cascade} by means of a smooth reparametrization $\theta \mapsto (\beta_{ij}(\theta))$ and a smooth ``decoding map'' $\varphi(w,z) = x$ that converts the internal weights $w \in \Reals^d$ and internal activations $z \in \Reals^n$ into the overall state $x \in \Reals^n$. Observe that the transformed parameters $\beta_{ij}(\theta)$ enter into \eqref{eq:KL_cascade_weights} linearly and that the vector fields $b_1,\dots,b_d$ and the maps $h$ and $\varphi$ do not have any dependence $\theta$. Moreover, since we have not imposed any restrictions on the drift vector field $f$ (apart from the regularity conditions needed to ensure existence and uniqueness of Stratonovich solutions), the Lie algebra generated by the vector fields $h(\cdot,w)$ may be infinite-dimensional.}
\end{remark}
\begin{proof} We closely follow the proof of Theorem~1 in \cite{krener1981complexity} (the special case of Abelian $\cA$ was considered in an earlier paper by Freedman and Willems \cite{freedman1978smooth}). Let $\tilde{g}_1,\ldots,\tilde{g}_d$ be a basis of $\cA$, so that
	\begin{align}\label{eq:g_i_expanded}
		g_i(x;\theta) = \sum^d_{j=1} \beta_{ij}(\theta) \tilde{g}_j(x), \quad i = 1,\dots,m
	\end{align}
	for some smooth maps $\beta_{ij} : \Reals^k \to \Reals$. Let
	\begin{align*}
		\varphi(w,z) \deq e^{w_1 \tilde{g}_1} \circ e^{w_2 \tilde{g}_2} \circ \dots \circ e^{w_d \tilde{g}_d}z,
	\end{align*}
	where $e^{t \tilde{g}_i}$ denotes the flow map of $\tilde{g}_i$, i.e., $e^{t\tilde{g}_i}z$ is the solution, at time $t$, of the ODE
	\begin{align*}
		\dot{z}(t) = \tilde{g}_i(z(t)), \qquad z(0) = z.
	\end{align*}
We have $\varphi(0,x) = x$. Using the methods of \cite{krener1981complexity}, we can show that there exist smooth vector fields $b_1,\dots,b_d$ on $\Reals^d$ and a neighborhood $\cU$ of $(0,x)$, such that
\begin{align*}
	\frac{\partial \varphi}{\partial w}(w,z) b_i(w) = \tilde{g}_i(\varphi(w,z)), \qquad i = 1,\dots,m
\end{align*}
for all $(w,z) \in \cU$. Consequently,
\begin{align*}
\frac{\partial \varphi}{\partial w}(w,z) \bigg(\sum^d_{j=1}\beta_{ij}(\theta)b_i(w)\bigg) = g_i(\varphi(w,z); \theta)
\end{align*}
for all $1 \le i \le d$ and all $\theta \in \Reals^k$. Now, $\varphi(w,z)$ is equal to the solution, at $t = d$, of the time-inhomogeneous ODE
\begin{align*}
	\dot{\xi}(t) = v(\xi(t),t), \qquad z(0) = z
\end{align*}
with
\begin{align*}
	v(z,t) \deq w_{d-i+1} \tilde{g}_{d-i+1}(z), \qquad i-1 \le t < i
\end{align*}
and the Jacobian $\frac{\partial \varphi}{\partial z}(w,z)$ is equal to the solution, at $t=d$, of the variational equation
\begin{align*}
	\dot{\Lambda}(t) = \frac{\partial v}{\partial z}(\xi(t),t)\Lambda(t), \qquad \Lambda(0) = I_n.
\end{align*}
Hence, it is invertible, so we take
\begin{align*}
	h(z,w) \deq \bigg(\frac{\partial \varphi}{\partial z}(w,z)\bigg)^{-1} f(\varphi(w,z))
\end{align*}
for all $(w,z) \in \cU$. Let $(W_t,Z_t)$ be the solution of \eqref{eq:KL_cascade} with $(W_0,Z_0) = (0,x)$, and let $\tau$ be the first exit time from $\cU$:
\begin{align*}
	\tau  \deq \inf \{ t > 0 : (W_t,Z_t) \not\in \cU \}.
\end{align*}
Then, since Stratonovich differentials obey the rules of ordinary calculus, for $0 \le t < \tau$ we have
\begin{align*}
	\dif X_t &= \dif \varphi(W_t,Z_t) \\
	&= \frac{\partial \varphi}{\partial w}(W_t,Z_t) \circ \dif W_t + \frac{\partial \varphi}{\partial z}(W_t,Z_t) \circ \dif Z_t \\
	&= f(\varphi(W_t,Z_t)) \dif t + \sum^m_{i=1} g_i(\varphi(W_t,Z_t); \theta) \circ \dif V^i_t \\
	&= f(X_t) \dif t + \sum^m_{i=1} g_i(X_t; \theta) \circ \dif V^i_t
\end{align*}
and $v^\trn \varphi(W_t,Z_t) = v^\trn X_t = Y_t$. This gives the desired simulation property.
\end{proof}

\section*{Acknowledgments}

T.~Veeravalli would like to thank A.J.~Havens for a PyTorch crash course. M.~Raginsky would like to thank A.~Belabbas for useful suggestions in the early stages of this work.

\bibliography{constructive_neural_SDE.bbl}

\end{document}